\documentclass[11pt]{amsart}

\usepackage[all]{xy}

\usepackage{amsmath, amssymb, amsthm, latexsym, amscd, mathrsfs}

\newtheorem{theorem}[subsection]{Theorem}
\newtheorem{prop}[subsection]{Proposition}
\newtheorem{lemma}[subsection]{Lemma}

\theoremstyle{definition}

\newtheorem{remark}[subsection]{Remark}

\newcommand{\N}{\mathbb N}
\newcommand{\Q}{\mathbb Q}
\newcommand{\C}{\mathbb C}
\newcommand{\Z}{\mathbb Z}
\newcommand{\M}{\mathbb M}

\title{regular representations of the quantum groups at roots of unity}
\author{Minxian Zhu}
\address{Department of Mathematics, Yale University, New Haven, CT 06520}
\email{minxian.zhu@yale.edu}

\begin{document}
\maketitle

\begin{abstract}

We study the bimodule structure of the quantum function algebra at roots of $1$ 
and prove that it admits an increasing filtration with factors isomorphic to 
the tensor products of the dual of Weyl modules $V_\lambda^* \otimes V_{- \omega_0 \lambda}^*$. 
As an application we compute the $0$-th Hochschild cohomology of the function algebra at roots of $1$.

\end{abstract}

\section{Introduction}

Let $\mathfrak g$ be a simple complex Lie algebra, 
and let $\mathbf U$ be the quantized enveloping algebra of $\mathfrak g$ 
over $\Q(v)$, $v$ an indeterminate. 
Let $\mathbf O$ be the linear span of matrix coefficients of finite-dimensional $\mathbf U$-modules (see [D]), 
then there is a perfect Hopf algebra pairing between $\mathbf U$ and $\mathbf O$. 
Moreover, 
as a $\mathbf U \otimes \mathbf U$-module, 
$\mathbf O$ has a classical Peter-Weyl type of decomposition.  
Following [L3], 
set $\mathscr A = \Z[v, v^{-1}]$
and let $U$ denote Lusztig's $\mathscr A$-form 
of $\mathbf U$ generated by the divided powers. 
Set $\mathscr A_0 = \Q [v, v^{-1}]$ and 
$U_{\mathscr A_0} = U \otimes_{\mathscr A}  \mathscr A_0$. 
Let $U_{\mathscr A_0}^*$ be the set of all $\mathscr A_0$-linear maps 
$U_{\mathscr A_0} \to \mathscr A_0$, 
and set $O = \mathbf O \cap U_{\mathscr A_0}^*$. 
Then $O$ is a Hopf algebra over $\mathscr A_0$ and 
the inclusion $O \subset \mathbf O$ induces an isomorphism of $\Q(v)$-algebras: 
$O \otimes_{\mathscr A_0} \Q(v) \cong \mathbf O$. 
Now let $q\in \C$ be a primitive $\ell$-th root of unity; 
set $U_q = U_{\mathscr A_0} \otimes_{\mathscr A_0} \Q(q)$ 
and $O_q  = O \otimes_{\mathscr A_0} \Q(q)$, 
where $\Q(q)$ is made into an $\mathscr A_0$-algebra by specializing $v$ to $q$. 
There is a Hopf algebra pairing between $U_q$ and $O_q$, 
thus $O_q$ admits the structure of a $U_q \times U_q$-module. 
The main goal of the present paper is to investigate this bimodule structure. 

Our motivation comes from a family of vertex operator algebras associated to
the modified regular representations of the affine Lie algebra $\hat {\mathfrak g}$ (see [Z1] and references therein). 
Each one of these vertex operator algebras admits two commuting actions of 
$\hat {\mathfrak g}$ in dual levels. 
When the dual central charges are generic, 
it decomposes into summands corresponding to the dominant weights of $\mathfrak g$. 
However when the dual central charges are rational,  
the module structure is less well understood, 
and it should be closely related to the regular representation 
of the corresponding quantum group at a root of unity, 
by the equivalence of tensor categories 
established by Kazhdan and Lusztig
between representations of the affine Lie algebra and representations of the quantum group (see [KL1-4]).

For simplicity, we assume $\mathfrak g = \mathfrak {sl}_n$ is of type $A$, and $\ell \geq n$ is odd.   
The quantum coordinate algebra $\mathbf O$ of $ \mathfrak {sl}_n$ can be described by generators and relations. 
Let $V$ be the quantization of the $n$-dimensinal natural representation of $\mathfrak {sl}_n$, 
then $\mathbf O$ is generated by the matrix coefficients $X_{ij}, 1\leq i, j\leq n$ of $V$ subject to a list of relations (see [D], [T], [APW]). 
In Section 2, we will show that $O$, as an $\mathscr A_0$-subalgebra of $\mathbf O$, 
is generated by $X_{ij}$'s over $\mathscr A_0$, 
which generalizes Proposition 1.3 of [CL]. 
A similar result was obtained in the Appendix of [APW] by P. Polo, 
using some local ring as the basic ring. 
Another $\mathscr A_0$-form of $\mathbf U$, 
denoted by $\Gamma (\mathfrak g)$ and slightly different from $U_{\mathscr A_0}$, 
was introduced in [CL]. 
The dual of $\Gamma (\mathfrak g)$, appropriately defined, coincides with $O$.
Specializing $v$ to $q$, a primitive $\ell$-th root of $1$, 
we get a perfect pairing between 
$\Gamma_q (\mathfrak g) = \Gamma (\mathfrak g) \otimes_{\mathscr A_0} \Q(q)$ and 
$O_q = O \otimes_{\mathscr A_0} \Q(q) $,  
hence it induces an embedding $O_q \hookrightarrow \Gamma_q (\mathfrak g) ^*$ ([CL, Lemma 6.1]). 
When $U_q = U_{\mathscr A_0} \otimes_{\mathscr A_0} \Q(q)$ 
and $O_q$ are concerned, 
we still have the embedding $O_q \hookrightarrow U_q^*$, 
though the pairing between $U_q$ and $O_q$ is in general degenerate on the $U_q$-side.  

Consider the non-semisimple category $\mathscr C_f$ of finite dimensional $U_q$-modules (of type $\mathbf 1$). 
The quantum function algebra $O_q$ 
can be realized as the linear span of matrix coefficients of modules from $\mathscr C_f$. 
One of our main results is that $O_q$, 
as a $U_q \times U_q$-module, 
admits an increasing filtration with factors isomorphic to the tensor products of the dual of Weyl modules 
$V_\lambda^* \otimes V_{- \omega_0 \lambda}^*$. 
It can be regarded as a generalization of the Peter-Weyl type of decomposition for $O_q$, 
$q$ not a root of unity.
Similar results for the regular representation of the affine Lie algebra 
in a rational level 
provide a proof of the conjecture, 
stated at the end of [Z1], 
about the bimodule structure of a family of vertex operator algebras in rational levels (see [Z2]). 
As an application of this increasing filtration of $O_q$, 
we show that the cocommutative elements of $O_q$ 
are linear combinations of the ``traces" of modules from $\mathscr C_f$, 
moreover as an algebra, 
it is isomorphic to the Grothendieck ring of $\mathscr C_f$ extended to the field $\Q(q)$. 

The paper is organized as follows: 
Section 2 gives an overview of the quantized enveloping algebra $\mathbf U$, 
the quantum coordinate algebra $\mathbf O$, 
Lusztig's $\mathscr A$-form $U$ of $\mathbf U$, 
the corresponding $\mathscr A_0$-form $O$ of $\mathbf O$, 
and their specializations $U_q, O_q$ at a root of unity.  
In particular we describe $O$ for type $A$ (Proposition \ref{O(sl_n)}),  
and identify $O_q$ with the matrix coefficients of finite dimensional $U_q$-modules
(Proposition \ref{O_q}). 
In Section 3, we describe an increasing filtration of $O_q$ using tilting modules 
(Theorem \ref{maintheorem2}), 
and compute the $0$-th Hochschild cohomology of $O_q$ as a coalgebra 
(Proposition \ref{Hochschild}). 
We treat the $\mathfrak {sl}_2$ case more thoroughly in Section 4 
and are able to obtain more explicit results 
(Theorem \ref{maintheorem1}).

I am very grateful to my advisor Igor Frenkel for his guidance, 
and Catharina Stroppel for helpful discussions.

\section{The General Setting}

Let $(a_{ij})_{1\leq i, j \leq n-1}$ be the Cartan matrix of a simply-laced simple Lie algebra $\mathfrak g$. 
The quantized enveloping algebra $\mathbf U$ 
is the $\Q(v)$-algebra defined by 
the generators $E_i, F_i, K_i, K_i^{-1} (1 \leq i \leq n-1)$ 
and the relations 
$$K_i K_i^{-1}= K_i^{-1} K_i =1, \qquad K_iK_j = K_jK_i,$$
$$K_i E_j = v^{a_{ij}} E_j K_i, \qquad K_i F_j = v^{-a_{ij}} F_j K_i,$$
$$E_iF_j - F_j E_i = \delta_{ij} \frac{K_i- K_i^{-1}}{v-v^{-1}},$$
$$E_iE_j = E_j E_i, \qquad F_iF_j = F_j F_i, \qquad \textrm{if      } a_{ij}=0,$$
$$E_i^2  E_j -(v+v^{-1}) E_iE_jE_i + E_j E_i^2 = 0, \qquad \textrm{if     } a_{ij} = -1,$$
$$F_i^2  F_j - (v+v^{-1}) F_iF_jF_i + F_j F_i^2 = 0, \qquad \textrm{if     } a_{ij} = -1. $$

$\mathbf U$ is a Hopf algebra over $\Q(v)$ 
with comultiplication $\triangle$, counit $\varepsilon$ and antipode $S$ defined by 
$$\triangle (E_i) = E_i \otimes 1 + K_i \otimes E_i, \qquad 
\triangle (F_i) = F_i \otimes K_i^{-1} + 1\otimes F_i, \qquad 
\triangle (K_i) = K_i \otimes K_i, $$
$$\varepsilon (E_i) = \varepsilon (F_i) =0, \qquad
\varepsilon (K_i) =1, $$
$$S (E_i) = - K_i^{-1} E_i, \qquad 
S(F_i) = - F_i K_i, \qquad 
S(K_i) = K_i^{-1}. $$

Following [L3, Sect 7], 
let $\mathscr F$ be the set of all two-sided ideals $I$ in $\mathbf U$ such that 
$I$ has finite codimension 
and there exists some $r\in \N$ such that 
for any $i$ we have $\prod_{h = -r}^{r} (K_{i} - v^{h}) \in I$. 
Let $\mathbf O$ be the set of all $\Q(v)$-linear maps $f: \mathbf U \to \Q(v)$ 
such that $f|_{I} = 0$ for some $I\in \mathscr F$. 
We call $\mathbf O$ the quantum coordinate (or function) algebra, 
which is equivalent to the linear span of matrix coefficients of finite dimensional $\mathbf U$-modules 
with a weight decomposition. 
Moreover $\mathbf O$ is a Hopf algebra over $\Q(v)$, 
and there exists a perfect Hopf algebra pairing 
$\mathbf U \times  \mathbf O \to \Q(v)$. 
Since  all finite dimensional $\mathbf U$-modules are completely reducible, 
$\mathbf O$ has a classical Peter-Weyl type of decomposition as a $\mathbf U \times \mathbf U$-module. 

For type $A$, we can describe $\mathbf O$ by generators and relations. 
Set $\mathfrak g = \mathfrak {sl}_n$, 
and let $(a_{ij})_{1\leq i, j \leq n-1}$ be the Cartan matrix with 
$a_{ij} = -1$ if $| i - j | =1$; $2$ if $i=j$; $0$ otherwise.
Let $\alpha_1, \cdots, \alpha_{n-1}$ be the simple roots of $\mathfrak {sl}_n$ 
associated to $(a_{ij})_{1\leq i, j\leq n-1}$,  
and let $\omega_1, \cdots, \omega_{n-1}$ be the corresponding fundamental weights. 
Let $V$ denote the quantization of the $n$-dimensional natural representation of $\mathfrak {sl}_n$ 
with highest weight $\omega_1$. 
Fix a highest weight vector $x_1 \in V$, 
and set $x_{i+1} = F_i x_i$ for all $1\leq i \leq n-1$. 
Then $x_i$ has weight $\omega_i - \omega_{i-1}$ 
(with the convention $\omega_0 = \omega_n =0$), 
and $\{ x_1, \cdots, x_n\}$ form a $\Q(v)$-basis of $V$
with $E_i x_{i+1} = x_i$. 
Let $\{\delta_1, \cdots, \delta_n\}$ be the dual basis in $V^*$, 
and define $X_{ij} \in \mathbf U^*$ 
by $X_{ij} (u) = \delta_i ( u \cdot x_j)$ for any $u\in \mathbf U$. 
These functionals $X_{ij}$ belong to $\mathbf O$ and they satisfy the following relations: 
$$X_{il} X_{jl} - v X_{jl} X_{il} =0 \qquad \text{for all    } l, i<j$$
$$X_{li} X_{lj} - v X_{lj} X_{li} =0 \qquad \text{for all    } l, i<j$$
$$X_{li} X_{mj} - X_{mj} X_{li} =0 \qquad \text{if    } l<m \text{  and    } i>j$$
$$X_{li} X_{mj} - X_{mj} X_{li} - (v - v^{-1}) X_{lj} X_{mi} = 0 \qquad \text{if    } l<m \text{  and    } i<j$$
$$\sum_{\sigma \in S_n} (-v)^{ l (\sigma) } X_{\sigma (1) 1} \cdots X_{ \sigma(n) n} = 1$$
where $l (\sigma)$ denotes the length of the permutation. 
In fact $\mathbf O$ is generated by $X_{ij}, 1\leq i, j \leq n$ subject to the above relations (see [D], [T]). 

Set $\mathscr A = \Z [v, v^{-1}]$. 
Given $n\in \Z$, $m\in \N$, we define
$ [ n ] = \frac{v^n - v^{-n}} { v- v^{-1}} \in \mathscr A $, 
$ [ m] ! = [ m ] [ m - 1 ] \cdots [1] $ and 
$  \left[ \begin{array}{c} n\\ m \end{array} \right] 
= \prod_{j=1}^{m} \frac{ v^{n-j+1} - v^{- n+j -1}} {v^j - v^{-j}} \in \mathscr A $.
Following [L1, L3], 
let $U$ be the $\mathscr A$-subalgebra of $\mathbf U$
generated by the elements $E_i^{(N)} = E_i^N / [N]!$, $F_i^{(N)} = F_i^N / [N]!$, 
$K_i$, $K_i^{-1}$ ($1\leq i \leq n-1, N \geq 0$). 
Then $U$ is a free $\mathscr A$-module 
and is itself a Hopf algebra over $\mathscr A$ in a natural way. 
Let $U^+, U^-, U^0$ be the $\mathscr A$-subalgebras of $U$ 
generated by the elements $E_i^{(N)}$; $F_i^{(N)}$; 
$K_i^{\pm 1}$, $\left[ \begin{array}{c} K_i; c\\t \end{array} \right]$,  
where
$$\left[ \begin{array}{c} K_i; c\\t \end{array} \right] 
= \prod_{s=1}^{t} \frac{ K_i v^{c -s+1} - K_i^{-1} v^{ -c +s -1}} { v^s - v^{-s}}, $$
then multiplication induces an isomorphism of $\mathscr A$-modules: 
$U^- \otimes U^0 \otimes U^+ \cong U$. 

We follow [L1] to construct an $\mathscr A$-basis of $U$. 
Let $s_i, 1\leq i \leq n-1$ be the simple reflections in the Weyl group $\mathcal W = S_n$ of $\mathfrak {sl}_n$. 
Let $R$, $R^+$ denote the root system and positive roots. 
Set $\alpha_{ij} = s_j s_{j-1} \cdots s_{i+1} \alpha_i = \sum_{k=i}^j \alpha_k \in R^+$ 
for any $1\leq i < j \leq n-1$, 
and consider the following total order on $R^+$: 
$\alpha_{n-1} < \alpha_{n-2, n-1} <  \cdots < \alpha_{1, n-1} 
<\alpha_{n-2} < \alpha_{n-3, n-2} < \cdots < \alpha_{1, n-2} 
< \cdots < \alpha_2 < \alpha_{12} < \alpha_1$. 
Let $\Omega: \mathbf U \to \mathbf U^{ \text{opp} }$ 
be the $\Q$-algebra isomorphism and 
$T_i: \mathbf U \to \mathbf U$, $1\leq i \leq n-1$ 
be the $\Q(v)$-algebra isomorphism defined in [L1, Sect 1]. 
Set $E_{ij} = T_j T_{j-1} \cdots T_{i+1} E_i$, 
and define for any $\phi, \phi' \in \N^{R^+}$, 
$$E^\phi = \prod_{\beta \in R^+} E_{\beta}^{( \phi (\beta) )}, \qquad F^{\phi'} = \Omega (E^{\phi'}), $$
where $E_{\alpha_i} = E_i$, $E_{\alpha_{ij}} = E_{ij}$, 
$E_\beta^{(N)} = E_\beta^N / [N]!$ 
and the factors in $E^\phi$ are written in the given order of $R^+$. 
Then the elements $E^\phi$; $F^{\phi'}$; 
$\prod_{i=1}^{n-1} K_i^{\delta_i} \left[ \begin{array}{c} K_i; 0\\t_i \end{array} \right] $, 
$t_i \geq 0, \delta_i = 0 \text{  or  } 1$
form an $\mathscr A$-basis of $U^+$; $U^-$; $U^0$ respectively. 
Hence the elements $F^{\phi'} K  E^\phi $, 
with $K$ in the above $\mathscr A$-basis of $U^0$, 
form an $\mathscr A$-basis of $U$. 
They also form a $\Q(v)$-basis of $\mathbf U$, 
hence we have an isomorphism of $\Q(v)$-algebras: 
$U \otimes_{\mathscr A} \Q(v) \cong \mathbf U$. 

Again following [L3, Sect 7], 
set $\mathscr A_0 = \Q [v, v^{-1}]$ 
and $U_{\mathscr A_0} = U \otimes_{\mathscr A} \mathscr A_0$. 
Let $U_{\mathscr A_0}^*$ be the set of all $\mathscr A_0$-linear maps 
$U_{\mathscr A_0} \to \mathscr A_0$
and let $O = \mathbf O \cap U_{\mathscr A_0}^*$. 
Then $O$ is a Hopf algebra over $\mathscr A_0$, 
and the inclusion $O \hookrightarrow \mathbf O$
induces an isomorphism of Hopf $\Q(v)$-algebras: 
$ O \otimes_{\mathscr A_0} \Q(v) \cong \mathbf O$. 
Let $\mathscr M$ be a $U_{\mathscr A_0}$-module, 
which is a free $\mathscr A_0$-module of finite rank 
with a basis in which the operators 
$K_i, \left[ \begin{array}{c} K_i; 0\\t_i \end{array} \right]$ 
act by diagonal matrices with eigenvalues 
$v^m, \left[ \begin{array}{c} m \\ t \end{array} \right]$. 
For any $m \in \mathscr M$ 
and $\xi \in \text{Hom}_{\mathscr A_0} ( \mathscr M, \mathscr A_0 )$, 
the matrix coefficient 
$c_{m, \xi}: u \to \xi (u \cdot m)$, an element of $U_{\mathscr A_0}^*$, belongs to $O$. 
Moreover $O$ is exactly the $\mathscr A_0$-submodule of $U_{\mathscr A_0}^*$ 
spanned by the matrix coefficients $c_{m, \xi}$ for various $\mathscr M, m, \xi$ as above. 

Another integral form of $\mathbf U$, denoted by $\Gamma (\mathfrak g)$, 
was introduced in [CL]. 
By definition $\Gamma (\mathfrak g)$ is 
the $\mathscr A_0$-subalgebra of $\mathbf U$ 
generated by $E_i^{(N)}, F_i^{(N)}, K_i ^{\pm 1}, \left( \begin{array}{c} K_i; c\\t \end{array} \right) $, 
where 
$ \left( \begin{array}{c} K_i; c\\t \end{array} \right) = \prod_{s=1}^t \frac{ K_i v^{c-s+1} -1} {v^s -1}$.
This algebra is larger than $U_{\mathscr A_0}$ because
its Cartan part $\Gamma (\mathfrak t)$ is larger than that of  $U_{\mathscr A_0}$. 
The elements 
$\prod_{i=1}^{n-1} \left( \begin{array}{c} K_i; 0 \\ t_i \end{array} \right) K_i ^{ - [t_i /2] }$ 
($ t_i \geq 0$)  
form an $\mathscr A_0$-basis of $\Gamma (\mathfrak t)$, 
where the Gauss symbol $[x]$ denotes the largest integer that is not greater than $x$. 
Let $\mathscr C$ be the full subcategory of $\Gamma (\mathfrak g)$-modules, 
which is free of finite rank as an $\mathscr A_0$-module and has 
a basis in which the operators $K_i, \left( \begin{array}{c} K_i; 0 \\ t \end{array} \right)$
act by diagonal matrices with eigenvalues $v^m, \left( \begin{array}{c} m \\ t \end{array} \right)$, 
where  $\left( \begin{array}{c} m \\ t  \end{array} \right) = \prod_{s=1}^t \frac{ v^{m-s+1} -1 }{ v^s -1 }$. 
Then the dual of $\Gamma (\mathfrak g)$, 
defined to be the linear span of matrix coefficients of modules from $\mathscr C$, 
coincides with $O$ ([CL, Remark 4.1]). 

Recall that for $\mathfrak g = \mathfrak {sl}_n$, 
the quantum coordinate algebra $\mathbf O$ is generated by 
$X_{ij}, 1\leq i, j \leq n$ over $\Q(v)$ subject to some relations. 
We want to show that its subalgebra $O$ 
is generated by $X_{ij}, 1\leq i, j\leq n$ over $\mathscr A_0$.  
It generalizes Proposition 1.3 of [CL], 
and the proof written below is pure calculation. 

\begin{prop} \label{O(sl_n)}
The $\mathscr A_0$-subalgebra $O$ of $\mathbf O$ is generated by $X_{ij}, 1\leq i, j \leq n$. 
\end{prop}

\begin{proof}
Let $\Xi$ be the  set of all matrices $M = (r_{ij})_{1\leq i, j\leq n}$ 
such that $r_{ij} \in \N$ and at least one of $r_{11}, \cdots, r_{nn}$ is zero. 
Fix a total order on $\{1, \cdots, n \}^2$, 
and set $X^M = \prod_{ij} X_{ij}^{ r_{ij} } $  for any $M \in \Xi$. 
Then $\{ X^M, M\in \Xi \}$ form a $\Q(v)$-basis of $\mathbf O$.  
Any element $f \in O$ can be represented uniquely as 
$f = \sum_{M \in \Xi} \gamma_M X^M$ with $\gamma_M \in \Q(v)$. 
Since $X_{ij} \in O$ for all $1 \leq i, j \leq n$, 
it suffices to prove that one of the nonzero coefficients $\gamma_M$ belongs to $\mathscr A_0$.   
Define a set of nonnegative integers inductively as follows: 
$$s_{n1} = \text{min} \{ r_{n1} | \gamma_{M = (r_{ij})} \neq 0 \}$$
$$s_{n-1, 1} = \text{min} \{ r_{n-1, 1} | \gamma_{M = (r_{ij})} \neq 0, r_{n1} = s_{n1} \}$$
$$\cdots$$
$$s_{21} = \text{min} \{ r_{21} | \gamma_{M = (r_{ij})} \neq 0, r_{i1} = s_{i1}, 2 < i \leq n \}$$
$$s_{n2} = \text{min} \{ r_{n2} | \gamma_{M = (r_{ij})} \neq 0, r_{i1} = s_{i1}, 1 < i \leq n \} $$
$$\cdots$$
$$s_{32} = \text{min} \{ r_{32} | \gamma_{M = (r_{ij})} \neq 0, r_{j2} = s_{j2}, r_{i1} = s_{i1}, 
3 < j \leq n, 1 < i \leq n \} $$
$$\cdots$$
$$s_{n, n-1} = \text{min} \{ r_{n, n-1} | \gamma_{M = (r_{ij})} \neq 0, r_{ij} = s_{ij}, 
1 \leq j < n-1,  j < i \leq n \}. $$
Define $\phi \in \N^{R^+}$; $\alpha_i \mapsto s_{i+1, i}, \alpha_{ij} \mapsto s_{j+1, i}$, 
then $f ( F^\phi K E ) = \sum_{M \in \Lambda} \gamma_M X^M ( F^\phi K E )$
for any $K \in U^0, E \in U^+$, 
where $\Lambda = \{ M = (r_{ij}) \in \Xi | \gamma_M \neq 0, r_{ij} = s_{ij}, 1 \leq j < i \leq n\}$. 
Similarly define another set of nonnegative integers: 
$$s_{1n} = \text{min} \{ r_{1n} | M = (r_{ij}) \in \Lambda \} $$
$$\cdots$$
$$s_{12} = \text{min} \{ r_{12} | M = (r_{ij}) \in \Lambda, r_{1i} = s_{1i}, 2 < i \leq n \}$$
$$s_{2n} = \text{min} \{ r_{2n} | M = (r_{ij}) \in \Lambda, r_{1i} = s_{1i}, 1 < i \leq n\}$$
$$\cdots$$
$$s_{23} = \text{min} \{ r_{23} | M = (r_{ij}) \in \Lambda, r_{2j} = s_{2j}, r_{1i} = s_{1i}, 3< j \leq n, 1 < i \leq n \} $$
$$\cdots$$
$$s_{n-1, n} = \text{min} \{ r_{n-1, n} | M = (r_{ij}) \in \Lambda, r_{ij} = s_{ij}, 
1 \leq i < n-1, i < j \leq n \}. 
$$
Define $\psi \in \N^{R^+}$; $\alpha_i \mapsto s_{i, i+1}, \alpha_{ij} \mapsto s_{i, j+1} $, 
then $f ( F^\phi K E^\psi ) = \sum_{M \in \Upsilon} \gamma_M X^M ( F^\phi K E^\psi )$
for any $K \in U^0$, 
where $\Upsilon = \{ M = (r_{ij}) \in \Lambda | r_{ij} = s_{ij}, 1 \leq i < j \leq n \}$. 
In other words, 
$\Upsilon$ is the collection of matrices $M \in \Xi$
whose non-diagonal entries are $s_{ij}$'s 
and the coefficient of $X^M$ in $f \in O$ is nonzero, 
moreover 
$f ( F^\phi K E^\psi ) = \sum_{M \in \Upsilon} \gamma_M v^{n_M} \chi_{\lambda + \mu_M} (K)$
for some $n_M \in \Z$, 
where $\lambda = \sum_{i > j} s_{ij} ( \omega_j - \omega_{j-1} ) 
+ \sum_{i < j} s_{ij} ( \omega_i - \omega_{i-1} )$ and 
$\mu_M = \sum_{i=1}^{n-1} (r_{ii} - r_{i+1, i+1} ) \omega_i$ 
for $M = (r_{ij}) \in \Upsilon$. 
Here the character $\chi_{\nu}: U^0 \to \mathscr A$ 
associated to a weight 
$\nu = ( \nu_i ), \nu_i = \langle \nu, \alpha_i^\vee \rangle$ 
is defined by 
$\chi_{\nu} ( K_i^{\pm 1} ) = v^{\pm \nu_i}$, 
$\chi_{\nu} ( \left[ \begin{array}{c} K_i; c \\ t \end{array} \right] ) = \left[ \begin{array}{c} \nu_i + c \\ t \end{array} \right]$.  
Note that $\mu_M = \mu_{M'}$ if and only if $M = M'$. 
Assume the diagonal entries of $M$ and $M'$ 
are $(r_{11}, \cdots, r_{nn})$ and $( r_{11}', \cdots, r_{nn}' )$ respectively, 
then $\mu_M = \mu_{M'}$ iff $r_{ii} - r_{i+1, i+1} = r_{ii}' - r_{i+1, i+1}'$, 
which is equivalent to $r_{ii} - r_{ii}' = c, 1 \leq i \leq n$ for some constant $c$. 
Since $r_{ii}, r_{ii}' \in \N$ and at least one from each group is zero, 
$c$ must be zero. 
Now that the characters $\chi_{\lambda + \mu_M}, M \in \Upsilon$ are all different, 
there exist $K' \in U^0, M_0 \in \Upsilon$ 
such that $\chi_{\lambda + \mu_M} (K') = 1$
if $M = M_0$; $0$ otherwise. 
Hence $f ( F^\phi K' E^\psi ) = \gamma_{M_0} v^{n_{M_0}} \in \mathscr A_0$, 
hence $\gamma_{M_0} \in \mathscr A_0$. 
\end{proof}

Let $\ell \geq n$ (the Coxeter number of $\mathfrak{sl}_n$) be an odd integer,
and let $q$ be a primitive $\ell$-th root of $1$. 
Let $p_\ell (v)$ denote the $\ell$-th cyclotomic polynomial, 
then we have an isomorphism of fields $\mathscr A_0 / (p_\ell (v)) \cong \Q(q)$. 
Set $U_q = U_{\mathscr A_0} \otimes_{\mathscr A_0} \Q(q)$ and 
$O_q = O \otimes_{\mathscr A_0} \Q(q)$. 
They are both Hopf algebras over $\Q(q)$ 
and inherit the comultiplications, counits and antipodes 
from $U_{\mathscr A_0}$ and $O$ respectively.
We denote the images of 
$E_i^{(N)}, F_i^{(N)}, K_i^{\pm 1}, \left[ \begin{array}{c} K_i; c \\ t \end{array} \right] \in U_{\mathscr A_0}$
in $U_q$ by the same notations.

\begin{prop}
There is a pairing of Hopf algebras $(, ): U_q \times O_q \to \Q(q)$, 
and it induces an embedding $O_q \hookrightarrow U_q^*$. 
\end{prop}

\begin{proof}
This is basically Lemma 6.1 of [CL], 
where the specialization of the larger $\mathscr A_0$-algebra $\Gamma (\mathfrak g)$, 
i.e. $\Gamma_q = \Gamma (\mathfrak g) \otimes_{\mathscr A_0} \Q(q)$,  
is considered, 
and the pairing between 
$\Gamma_q$ and $O_q$ is non-degenerate. 
However if instead we consider the pairing between $U_q$ and $O_q$, 
it is only non-degenerate on the $O_q$-half, 
i.e. $ (u, f) = 0$ for all $u \in U_q$ implies that $f =0 \in O_q$. 
To see why it is degenerate on the $U_q$-half, 
consider the image of $K_i^\ell -1 \in U_{\mathscr A_0}$ in $U_q$, 
denoted by the same notation. 
It is not difficult to see that $( K_i^\ell -1, f ) = 0$ for all $f \in O_q$, 
but $K_i^\ell \neq 1$ in $U_q$ 
(instead $K_i^{2 \ell} = 1$ in $U_q$). 
The injectivity of the induced map $O_q \to U_q^*$ 
can also be proved using the arguments of Proposition \ref{O(sl_n)}. 
\end{proof}

The dual space $U_q^*$ admits two commuting (left) actions of $U_q$, 
which we denote by $\rho_1, \rho_2$. 
By definiton, 
$\rho_1 (u) f (u') = f( u' u)$ and 
$\rho_2 (u) f (u') = f( S(u) u')$
for any $f\in U_q^*, u, u' \in U_q$, 
where $S$ denotes the antipode of $U_q$. 
The Hopf algebra $O_q$ is a 
$U_q \times U_q$-submodule of $U_q^*$, 
and the $U_q$-actions can be expressed as follows: 
$$\rho_1 (u) g = \sum g_{(1)} (u, g_{(2)} ), \qquad 
\rho_2 (u) g = \sum (S(u), g_{(1)}) g_{(2)}, $$
for any $u\in U_q$, $g\in O_q$, 
where $\triangle (g) = \sum g_{(1)} \otimes g_{(2)}$. 
The question we want to investigate is 
how $O_q$ decomposes as a $U_q \times U_q$-module, 
i.e. as a bicomodule of itself. 

Let $V$ be a finite dimensional representation of $U_q$, and 
let $V^*$ be the dual representation defined by 
$u f (v ) = f (S(u) v)$ for any $u \in U_q, f \in V^*, v \in V$. 
It is obvious that the map 
$\phi_V: V \otimes V^* \to U_q^*; v \otimes f \mapsto f ( \cdot \, v )$ 
is a $U_q \times U_q$-morphism. 
We denote the image of $\phi_V$ by $\M(V)$, 
called the matrix coefficients of $V$. 
Usually $\phi_V$ is not injective unless $V$ is irreducible.

\begin{lemma}  \label{matrixcoefficients}
Let $V, V'$ be finite-dimensional $U_q$-modules and $U \subset V$ be a submodule, 
then we have the following: 
\begin{enumerate}
\item 
$\phi_V (U \otimes V^*) = \M (U)$ and 
$\phi_V ( V \otimes (V / U)^* ) = \M ( V / U)$. 
In particular $\M (U)$, $\M( V / U) \subset \M (V)$. 
\item
$\M ( V \otimes V' ) = \M (V) \cdot \M (V' )$ and 
$\M ( V \oplus V' ) = \M(V) + \M( V' )$. 
\item 
$\M ( V^* ) = S ( \M ( V) ) $. 
\end{enumerate}
\end{lemma}

\begin{proof}
(1) is easy to prove:  
in terms of matrix representations, 
$\M ( U )$ and $ \M ( V / U )$ are the matrix coefficients in the diagonal blocks. 
Also note that the multiplication and the map $S$ of $U_q^*$ 
are defined by taking the transposes of 
the comultiplication and the antipode of $U_q$.
Hence (2) and (3) follow. 
\end{proof}

We have a triangular decomposition 
$U_q = U_q^- U_q^0 U_q^+$, 
where $U_q^- = U^- \otimes_{\mathscr A} \Q(q)$ 
and similar definitions for $U_q^0$ and $U_q^+$. 
Denote by $X, X^+$ the weight lattice 
and the dominant weights of $\mathfrak g$. 
For $\lambda \in X$, 
the character $\chi_\lambda: U^0 \to \mathscr A$ 
induces a character of $U_q^0$ to $\Q(q)$: 
$K_i^{\pm 1} \mapsto q^{\pm \langle \lambda, \alpha_i^\vee \rangle}$, 
$\left[ \begin{array}{c} K_i; c \\ t \end{array} \right]  \mapsto \left[ \begin{array}{c} \langle \lambda, \alpha_i^\vee \rangle + c \\ t \end{array} \right]_q$, 
where the subscript $q$
means evaluating an element of $\Z[v, v^{-1}]$ at $v = q$. 
Let $\mathscr C_f$ be the category of 
finite dimensional $U_q$-modules
with a weight decomposition with respect to $U_q^0$. 
We will show that 
$O_q \subset U_q^*$ 
is precisely the linear span of matrix coefficients of modules from $\mathscr C_f$. 
To prove it, let's first recall some important modules in $\mathscr C_f$. 
For any dominant weight $\lambda \in X^+$, 
we can associate four canonical modules: 
the Weyl module $V_\lambda$, 
the dual of the Weyl module $V_\lambda^*$, 
the irreducible module $L_\lambda$
and the tilting module $T_\lambda$. 

The Weyl module $V_\lambda$ 
is generated by a vector of highest weight $\lambda$, 
and has the universal property that 
any module in $\mathscr C_f$ 
generated by a vector of highest weight $\lambda$ 
is a quotient of $V_\lambda$. 
The character of $V_\lambda$ 
is given by Weyl's character formula 
(see [APW, A2] for definitions of $V_\lambda$ and $V_\lambda^*$
in terms of some induction functor and its derived functors). 
We have the following property for these standard objects: 
$\text{Ext}_{\mathscr C_f}^i ( V_\lambda, V_\mu^* ) = \Q(q)$ 
if $i = 0$ and $\lambda = - \omega_0 \mu$; 
$0$ otherwise, where $\omega_0$ 
is the longest element in the Weyl group.  

The irreducible module $L_\lambda$ 
is the head of $V_\lambda$ 
as well as the socle of $V_{ - \omega_0 \lambda}^*$, 
and $L_\lambda^* \cong L_{ - \omega_0 \lambda}$. 
Furthermore the modules $L_\lambda, \lambda\in X^+$ 
give a complete list of non-isomorphic irreducible modules in $\mathscr C_f$. 

A module in $\mathscr C_f$ is called tilting 
if it admits both a Weyl filtration and a dual Weyl filtration. 
Tilting modules are closed by taking the dual and the tensor product. 
For each $\lambda \in X^+$, 
there exists a unique (up to isomorphism) 
indecomposable tilting module $T_\lambda$
such that $T_\lambda$ 
admits a Weyl filtration starting with $V_\lambda \hookrightarrow T_\lambda$, 
and any other Weyl modules $V_\mu$ 
entering the Weyl filtration of $T_\lambda$ satisfy that  $\mu < \lambda$, 
here $\leq$ is the usual partial order on $X$ 
determined by a set of positive roots. 
The highest weight $\lambda$ 
occurs with multiplicity $1$ in $T_\lambda$. 
By consideration of characters, 
we have  $T_\lambda^* \cong T_{ - \omega_0 \lambda }$. 
Moreover the modules $T_\lambda, \lambda \in X^+$
form a complete list of inequivalent indecomposable tilting modules. 
It is easy to see that $T_\lambda = V_\lambda$ 
if and only if $V_\lambda$ is irreducible. 
There are enough projectives in $\mathscr C_f$
and all projective modules are tilting (see [APW], [A2]). 




\begin{prop} \label{O_q}
$O_q \subset U_q^*$ is the linear span of matrix coefficients of modules from $\mathscr C_f$, 
i.e. $O_q = \sum_{V \in \mathscr C_f} \M ( V)$. 
\end{prop}

\begin{proof}
By Proposition \ref{O(sl_n)}, 
the $\Q(q)$-algebra $O_q$ 
is generated by $X_{ij}$, 
the matrix coefficients of $V_{\omega_1}$. 
By Lemma \ref{matrixcoefficients} (2), we have 
$O_q = \sum_n \M ( V_{\omega_1}^{\otimes n} ) \subset \sum_{V \in \mathscr C_f} \M ( V)$. 
To prove the inverse inclusion, it suffices to show that $O_q$ 
contains the matrix coefficients of all the tilting modules, 
since all projective modules are tilting. 

Note that the Weyl modules $V_{\omega_i}, i=1, \cdots, n-1$ 
associated to the fundamental weights are all irreducible, 
since the weights occurring in $V_{\omega_i}$ 
lie in the $W$-orbit of $\omega_i$ 
with multiplicity $1$, 
hence we must have $L_{\omega_i} = V_{\omega_i} = T_{\omega_i}$. 
It is well known that the fundamental representation of $\mathfrak {sl}(n, \C)$ 
with highest weight $\omega_i$ can be realized as the $i$-th exterior power of 
the $n$-dimensional natural representation, 
in particular it is a direct summand of the $i$-th tensor power.  
Therefore by consideration of the characters, 
$V_{\omega_i}$ is either a direct summand of $V_{\omega_1}^{\otimes i}$ 
or a composition factor of a direct summand of $V_{\omega_1}^{\otimes i}$
(which is tilting). 
Either way, we have $\M(T_{\omega_i}) \subset \M( V_{\omega_1} )^i$ 
by Lemma \ref{matrixcoefficients}. 
Now for any $\lambda \in X^+$ 
with $m_i = \langle \lambda, \alpha_i^{\vee} \rangle$, 
the tilting module $T_\lambda$ 
must be a direct summand of 
$T_{\omega_1}^{\otimes m_1} \otimes \cdots \otimes T_{\omega_{n-1}}^{\otimes m_{n-1}}$, 
hence $\M ( T_\lambda) \subset \M ( T_{\omega_1} )^{m_1} \cdots \M ( T_{\omega_{n-1}} )^{m_{n-1}} \subset \sum_n \M ( V_{\omega_1})^n$. 
\end{proof}

\section{an increasing filtration of $O_q$}

Fix $\mathfrak g = \mathfrak{sl}_n$, and $q$ to be a primitive $\ell$-th root of $1$, 
where $\ell \geq n$ is odd. 
It follows from Proposition \ref{O(sl_n)} that the quantum coordinate algebra $O_q$ 
is generated by $X_{ij}, 1 \leq i, j \leq n$ over $\Q(q)$ 
subject to a list of relations. 
Proposition \ref{O_q} identifies $O_q$ with 
the linear span of matrix coefficients of finite dimensional $U_q$-modules. 
In this section, we will describe a canonical increasing filtration of $O_q$ as a $U_q \times U_q$-module. 

Let $R, R^+, X, X^+, \mathcal W$ denote the root system, 
positve roots, weight lattice, dominant weights and Weyl group of $\mathfrak g$. 
The affine Weyl group $\mathcal W_\ell$ 
is generated by the affine reflections $s_{\beta, m}, \beta \in R^+, m \in \Z$ given by 
$$s_{\beta, m} \cdot \lambda = s_\beta \cdot \lambda + m \ell \beta, \quad \lambda \in X. $$
Here $s_\beta$ is the reflection corresponding to the positive root $\beta$, 
and we are using the dot-action defined by 
$s_\beta \cdot \lambda = s_\beta (\lambda + \rho) - \rho$, 
where $\rho$ is the half sum of the positive roots. 

Denote by $C$ the first dominant alcove, i.e. 
$$C = \{ \lambda \in X^+ | \, \langle \lambda + \rho, \beta^\vee \rangle < \ell \text{  for all } \beta \in R^+ \}, $$
and set 
$$\bar C = \{ \lambda \in X | \, 0 \leq \langle \lambda + \rho, \beta^\vee \rangle \leq \ell \text{  for all } \beta \in R^+ \}, $$
then $\bar C$ is a fundamental domain for the action of $\mathcal W_\ell$ on $X$. 

The linkage principal (see [A1]) allows us to decompose any module from $\mathscr C_f$
into summands corresponding to the representatives in $\bar C$, 
therefore it yields a decomposition of $O_q$ as well. 

\begin{prop}\label{link}
As a $U_q \times U_q$-module, we have 
$O_q \cong \bigoplus_{\lambda \in (\ell -1) \rho + \ell X^+} V_\lambda \otimes V_\lambda^* \oplus 
(\bigoplus_{\mu \in \bar C \setminus \{ (\ell -1 ) \rho + \ell X \} } \Lambda_\mu)$, 
where $\Lambda_\mu = \sum_{\nu \in \mathcal W_\ell \cdot \mu \cap X^+} \M (T_\nu)$. 
\end{prop}

\begin{proof} 
Recall that $O_q$ is spanned by the matrix coefficients of (tilting) modules from $\mathscr C_f$. 
The linkage principal implies that $O_q = \oplus_{\mu \in \bar C} \Lambda_\mu$, 
where $\Lambda_\mu = \sum_{\nu \in \mathcal W_\ell \cdot \mu \cap X^+} \M (T_\nu)$. 
The vertices of the simplex $\bar C$ are $ -\rho, \ell \omega_i - \rho, i =1, \cdots, n-1$, 
where $\omega_i$'s are the fundamental weights of $\mathfrak g$. 
The $\mathcal W_\ell$-orbits of these vertices consist of weights of the form $(\ell - 1)\rho + \ell X$. 
By [APW, Corollary 7.6], 
if $\lambda \in (\ell - 1) \rho + \ell X^+$, 
the Weyl module $V_\lambda$ is irreducible, 
in which case $V_\lambda = T_\lambda$
and $\M (T_\lambda) \cong V_\lambda \otimes V_\lambda^*$. 
\end{proof}

\begin{lemma} \label{bb}
Let $V, V' \in \mathscr C_f$. 
\begin{enumerate}
\item 
Suppose $V$ admits a Weyl filtration $0 = V^0 \subset V^1 \subset \cdots \subset V^m = V$ 
such that $V^i / V^{i-1} \cong V_{\lambda_i}$ for some $\lambda_i \in X^+$, 
then $\M (V) \subset \sum_i \M (T_{\lambda_i})$. 
\item 
Suppose $V'$ admits a dual Weyl filtration with factors isomorphic to $V_{\mu_i}^*$ 
for some $\mu_i \in X^+$, 
then $\M (V') \subset \sum_i \M (T_{\mu_i}^* )$. 
\end{enumerate}
\end{lemma}

\begin{proof}
Let $f_i$ be the composition of $V^i \twoheadrightarrow V^i / V^{i-1} \cong V_{\lambda_i} \hookrightarrow T_{\lambda_i}$. 
Apply $\text{Hom}_{\mathscr C_f} (-, T_{\lambda_i})$ to the short exact sequence
$0 \to V^i \to V \to V/ V^i \to 0$, 
we get 
$\text{Hom}_{\mathscr C_f} (V, T_{\lambda_i}) \to \text{Hom}_{\mathscr C_f} (V^i, T_{\lambda_i}) \to \text{Ext}_{\mathscr C_f} ( V/ V^i, T_{\lambda_i} ) $
which is exact. 
Since $V / V^i$ and $T_{\lambda_i}$ admit a Weyl filtration and a dual Weyl filtration respectively, 
it follows that $\text{Ext}_{\mathscr C_f} (V / V^i, T_{\lambda_i}) = 0 $, 
hence $\text{Hom}_{\mathscr C_f} (V, T_{\lambda_i}) \to \text{Hom}_{\mathscr C_f} (V^i, T_{\lambda_i})$ is surjective. 
Let $g_i: V \to T_{\lambda_i}$ be a preimage of $f_i$, 
then $ V^i \cap \text{Ker } g_i = V^{i-1}$. 
Define $g = \sum g_i: V \to \oplus_i T_{\lambda_i}$, 
then $\text{Ker } g = \cap_i \text{Ker } g_i = 0$, 
i.e. $g$ is injective. 
Hence by Lemma \ref{matrixcoefficients} we have $\M (V) \subset \sum_i \M (T_{\lambda_i})$. 
Analogously we can prove (2) 
by constructing a surjective map $\oplus_i T_{\mu_i}^* \to V'$,  
but we can also argue as follows: 
by assumption, $V'^*$ admits a Weyl filtration with factors isomorphic to $V_{\mu_i}$, 
hence by (1) we have $\M (V'^*) \subset \sum \M (T_{\mu_i})$, 
hence it follows from Lemma \ref{matrixcoefficients} (3) that $\M (V') = S (\M (V'^*)) \subset \sum_i S (\M (T_{\mu_i})) = \sum_i \M (T_{\mu_i}^*)$. 
\end{proof}

\begin{theorem} \label{maintheorem2}
Let $\mu \in \bar C \setminus \{ (\ell -1) \rho + \ell X \}$, and write 
$\mathcal W_\ell \cdot \mu \cap X^+ = \{ \nu_i, i \geq 1 \}$ so that 
$\nu_i \leq \nu_j$ implies $i \leq j$. 
Set $P^i = \sum_{j \leq i} \M (T_{\nu_j})$, 
then $P^1 \subset \cdots \subset P^{i-1} \subset P^i \subset \cdots$ 
is an increasing filtration of $U_q \times U_q$-submodules of $\Lambda_\mu$ 
with subquotients $P^i / P^{i-1} \cong V_{ - \omega_0 \nu_i}^* \otimes V_{\nu_i}^*$ as a $U_q \times U_q$-module. 
\end{theorem}

\begin{proof}
Since the dual Weyl filtration of  $T_{\nu_i}$ 
ends with $T_{\nu_i} \twoheadrightarrow V_{ - \omega_0 \nu_i }^*$, 
there exists a submodule $W \subset T_{\nu_i}$ 
such that $T_{\nu_i} / W \cong V_{ - \omega_0 \nu_i }^*$, 
and $W$ admits a filtration with factors isomorphic to $V_\gamma^*$'s 
with $ - \omega_0 \gamma < \nu_i$ 
and $- \omega_0 \gamma \in \mathcal W_\ell \cdot \nu_i$, 
i.e. $- \omega_0 \gamma = \nu_j$ for some $j < i$.
Hence we have $\phi_{T_{\nu_i}} ( W \otimes T_{\nu_i}^*) = \M (W) \subset P^{i-1}$
by Lemma \ref{matrixcoefficients} and Lemma \ref{bb}; 
analogously we also have 
$\phi_{T_{\nu_i}} ( T_{\nu_i} \otimes (T_{\nu_i} / V_{\nu_i} )^* ) 
= \M ( T_{\nu_i} / V_{\nu_i} ) \subset P^{i-1}$.
Set $N = W \otimes T_{\nu_i}^* + T_{\nu_i} \otimes (T_{\nu_i} / V_{\nu_i} )^*$, 
then $\phi_{T_{\nu_i}}$ induces a surjective map
$\psi: (T_{\nu_i} \otimes T_{\nu_i}^*) / N \twoheadrightarrow \M (T_{\nu_i}) / ( \M (T_{\nu_i}) \cap P^{i-1} ) = P^i / P^{i-1}$.
Note that 
$(T_{\nu_i} \otimes T_{\nu_i}^*) / N \cong V_{ - \omega_0 \nu_i}^* \otimes V_{\nu_i}^*$, 
the socle of which is $L_{\nu_i} \otimes L_{\nu_i}^*$. 
Since 
$\psi (L_{\nu_i} \otimes L_{\nu_i}^*) = (\M (L_{\nu_i}) + P^{i-1}) / P^{i-1} \cong \M (L_{\nu_i}) \neq 0$, 
$\psi$ is also injective,  
hence it induces the isomorphism 
$ V_{ - \omega_0 \nu_i}^* \otimes V_{\nu_i}^* \, \tilde \to \, P^i / P^{i-1}$. 
\end{proof}

As an application, we will compute $HH^0 (O_q, O_q)$, 
the $0$-th Hochschild cohomology of the coalgebra $O_q$ with coefficients in $O_q$, 
which is equivalent to the algebra of cocommutative elements in $O_q$. 

Suppose $f \in O_q$ is cocommutative, then 
$f (u u') = f (u' u)$ for any $u, u' \in U_q$, 
i.e. $\rho_1 (u) f = \rho_2 ( S^{-1} u ) f $. 

\begin{lemma} \label{dim1}
For any $\lambda \in X^+$,
the subspace 
$Y = \{ y \in V_{ - \omega_0 \lambda}^* \otimes V_\lambda^*: 
\rho_1' (u) y = \rho_2' ( S^{-1} u ) y, \, \forall u \in U_q \}$ 
is one-dimensional,  
where $\rho_1', \rho_2'$ denote the actions of $U_q$ 
on $V_{ - \omega_0 \lambda}^*$ and $V_\lambda^*$ respectively. 
\end{lemma}

\begin{proof}
Set $B_q = U_q^0 U_q^-$, $k = \Q(q)$, 
and denote by $k_\lambda$ the one-dimensional $B_q$-module
defined by the character $\chi_\lambda: U_q^0 \to k$ 
and extended to a $B_q$-module with trivial $U_q^-$-action. 
Recall from [APW, A1] that $V_{ - \omega_0 \lambda}^*$
is an integrable submodule of $\text{Hom}_{B_q} (U_q, k_\lambda)$, 
where $U_q$ is considered a $B_q$-module via left multiplication of $B_q$ on $U_q$, 
and the $U_q$-module structure on $\text{Hom}_{B_q} (U_q, k_\lambda)$ is defined 
via the right multiplication of $U_q$ on itself. 
Choose a basis $v_1, \cdots, v_s$ of $V_{- \omega_0 \lambda}^*$
such that $v_1 ( u E_i^{(r)} ) = 0$ for any $i$ if $r > 0$
and $v_1 (b) = \chi_\lambda (b)$ for any $b\in B_q$. 
Then $v_1$ has weight $\lambda$ 
(recall that $\lambda$ occurs with multiplicity $1$ in $V_{ - \omega_0 \lambda}^*$). 
Assume that $v_2, \cdots, v_s$ are also homogeneous vectors (with weights less than $\lambda$), 
then $v_i (b) = 0$ for any $b\in U_q^-$ and $i = 2, \cdots, s$. 
Similarly choose a homogeneous basis $v_1', \cdots, v_s'$ of $V_\lambda^*$ 
such that $v_1'$ has weight $ - \lambda$. 
For any $y = \sum_{ij} y_{ij} v_i \otimes v_j' \in Y$, 
it is easy to check that in order for $y$ 
to satisfy the equality $\rho_1' (u) y = \rho_2' ( S^{-1} u ) y$ 
for all $u = u^0 \in U_q^0$, 
we must have $y_{1i} = 0$ for any $i \neq 1$ 
(since $\chi_{- \lambda} S^{-1} = \chi_{ \lambda}$). 
Define a linear map $pr: Y \to k$; $y \mapsto y_{11}$, 
which we will show is in fact injective. 
Suppose $y_{11} = 0$ for some $y \in Y$, 
then for any $u^+_1, u^+_2 \in U_q^+$, 
we have 
$y ( u^+_1 \otimes u^+_2 ) = \sum_{ij} y_{ij} v_i ( u^+_1 ) v_j' ( u^+_2 )
=\sum_{ij} y_{ij}  \{ \rho_1' (u^+_1) v_i \} (1) v_j' ( u^+_2 )
=\sum_{ij} y_{ij} v_i (1) \{ \rho_2' ( S^{-1} u^+_1) v_j' \}  (u^+_2)
= y_{11} v_1' ( u^+_2 S^{-1} u^+_1 )
= 0$, 
hence $y = 0$. 
It means that $pr$ is injective, hence $\text{dim}_k Y \leq 1$. 
On the other hand, 
let $e_i$ be a basis of $L_\lambda$ 
and let $\delta_i$ be the dual basis of $L_\lambda^*$, 
it is easy to check that 
$\sum_i e_i \otimes \delta_i \in L_\lambda \otimes L_\lambda^* 
\subset V_{ - \omega_0 \lambda}^* \otimes V_\lambda^*$
satisfies the condition of $Y$, hence $\text{dim}_k Y = 1$. 
\end{proof}

\begin{prop} \label{Hochschild}
$HH^0 (O_q, O_q) \cong \Q(q) [ X ]^{\mathcal W}$. 
\end{prop}

\begin{proof}
For a module $V \in \mathscr C_f$, 
we denote by $[ V ]$ its image in the Grothendieck ring $[ \mathscr C_f ]$. 
It is clear that $[ \mathscr C_f ]$ is isomorphic to $ \Z [ X ]^{\mathcal W}$,
with the isomorphism given by $[ V ] \to \text{ch} V$. 
Let $\mathcal R = [ \mathscr C_f ] \otimes_\Z \Q(q)$, 
then $\mathcal R \cong \Q(q) [ X ]^{\mathcal W}$
and it has a natural basis of simple characters $\{ \text{ch} L_\lambda, \lambda \in X^+ \}$. 

For $V \in \mathscr C_f$, 
define the trace of $V$ as 
$tr_V = \phi_V ( \sum_i v_i \otimes f_i ) \in \M(V) $, 
where $\{ v_i \}$ is a basis of $V$ and $\{ f_i \}$ is the dual basis of $V^*$. 
Let $\mathbf {tr} \subset O_q$ be the $\Q(q)$-linear span of traces of modules from $\mathscr C_f$. 
If $U \hookrightarrow V \twoheadrightarrow W $ 
is a short exact sequence of modules from $\mathscr C_f$, 
we have $tr_V = tr_U + tr_W$, 
therefore each $tr_V$ can be written as a linear combination of traces of its composition factors. 
Note that $tr_\lambda (: = tr_{L_\lambda} )$, $\lambda \in X^+$, 
are linearly independent, 
hence they form a basis of $\mathbf {tr}$, 
and $\mathbf {tr} \cong \mathcal R$ as a vector space. 
Since $tr_{V \otimes V' } = tr_V tr_{V'}$, 
it is in fact an isomorphism of algebras. 

Denote by $\mathbf{Co}$ the set of elements in $O_q$ that are cocommutative, 
it suffices to show that $\mathbf {Co} = \mathbf {tr}$. 
It is obvious that $\mathbf {tr} \subset \mathbf {Co}$. 
To prove the inverse inclusion, 
define $P^\lambda = \sum_{\mu \leq \lambda, \mu \in X^+} \M (T_\mu) \subset O_q$. 
Since $O_q = \bigcup_{\lambda \in X^+} P^\lambda$, 
it suffices to prove that $P^\lambda \cap \mathbf {Co} \subset \mathbf {tr}$. 
If $\lambda$ is minimal (for the ordering $ \leq $) among the weights in $X^+$, 
then $P^\lambda = \M (T_\lambda) = \M (L_\lambda) \cong L_\lambda \otimes L_\lambda^*$. 
It follows from Lemma \ref{dim1} that 
$P^\lambda \cap \mathbf{Co} = \Q(q) tr_\lambda \subset \mathbf {tr}$. 
Now assume that $P^\mu \cap \mathbf {Co} \subset \mathbf{tr}$ 
is true for any $\mu < \lambda, \mu \in X^+$. 
From the proof of Theorem \ref{maintheorem2}, 
we have 
$P^\lambda / \sum_{\mu < \lambda, \mu \in X^+} \M (T_\mu) 
\cong V_{ - \omega_0 \lambda}^* \otimes V_\lambda^*$. 
Suppose $f \in P^\lambda \cap \mathbf {Co}$, 
then $\rho_1 (u) f = \rho_2 ( S^{-1} u ) f $ for any $u \in U_q$, 
hence the image of $f $ 
in $P^\lambda / \sum_{\mu < \lambda, \mu \in X^+} \M (T_\mu) $
belongs to the subspace $Y$ defined in Lemma \ref{dim1}. 
Since $Y$ is one-dimensional and is spanned by the image of the trace of $L_\lambda$, 
there exists a scalar $\zeta$ 
such that $f - \zeta tr_\lambda \in \sum_{\mu < \lambda, \mu \in X^+} \M (T_\mu) \cap \mathbf {Co}$. 
By induction $f - \zeta tr_\lambda \in \mathbf {tr}$, 
hence $f \in \mathbf {tr}$. 
\end{proof}

It is well known that the category of finite dimensional representations of $U_q$
is semisimple when $q$ is not a root of unity, 
in which case the quantum function algebra $O_q$ is the direct sum of matrix coefficients 
of irreducible modules, 
and all the cocommutative elements of $O_q$ come from 
the traces of finite dimensional modules. 
Proposition \ref{Hochschild} says that 
the last statement is also true at roots of $1$. 

\begin{remark}
For other types of simple Lie algebras, 
I am not sure if $O_q$ is linearly spanned 
by the matrix coefficients of finite dimensional $U_q$-modules.  
Nonetheless if we denote the latter by $O_q'$, 
then obviously $O_q \subset O_q'$, 
and the results in this section hold for $O_q'$. 
\end{remark}

\section{the case of $\mathfrak {sl}_2$}

In this section we study the $\mathfrak{sl}_2$ case more thoroughly. 
Let $\ell > 2 $ be odd, $q$ be a primitive $\ell$-th root of unity. 
The quantum function algebra $O_q$ 
is generated by $a, b, c, d$ over $\Q(q)$ 
subject to the relations: 
$$ab = q ba, \qquad ac = qca, $$
$$bd = qdb,  \qquad cd = qdc,  $$
$$bc=cb, \qquad  ad-qbc = da -q^{-1} bc =1. $$
The comultiplication $\triangle$, 
counit $\varepsilon$ and antipode $S$ are defined by
$$\triangle (a) = a \otimes a + b\otimes c, \qquad \triangle (b) = a \otimes b + b \otimes d, $$
$$\triangle (c) = c \otimes a + d \otimes c, \qquad \triangle (d) = c \otimes b + d \otimes d, $$
$$\varepsilon (a) = \varepsilon (d) =1, \qquad \varepsilon (b) = \varepsilon (c) = 0, $$
$$S(a) = d,  \quad S (d) = a, \quad S ( b) = - q^{-1} b, \quad S (c) = -q c. $$
The quantum group $U_q$
is generated by $E^{(i)}, F^{(i)}, K^{\pm 1},  \left[ \begin{array}{c} K;  c \\ t \end{array} \right]$
subject to some relations.

For $\mathfrak g = \mathfrak{sl}_2$, 
we have $X = \Z, X^+ = \N$. 
The Weyl module $V_n$, for $n \in \N $, 
is $(n+1)$-dimensional, with a basis $f_0, f_1, \cdots, f_n$ 
such that $f_i$ is of weight $-n+2i$ and 
$$E^{(j)} f_i = \left[ \begin{array}{c} i+j \\i \end{array} \right]_q f_{i+j}, \qquad
F^{(j)} f_i = \left[ \begin{array}{c} n - i+j  \\j \end{array} \right]_q f_{i-j}. $$
The dual representation $V_n^*$ is also $(n+1)$-dimensional, 
with a basis $e_0, e_1, \cdots, e_n$ such that $e_i$ is of weight $n- 2i$ 
and 
$$E^{(j)} e_i =  \left[ \begin{array}{c} i \\ j  \end{array} \right]_q e_{i-j}, \qquad 
F^{(j)} e_i =  \left[ \begin{array}{c} n-i \\ j \end{array} \right]_q e_{i+j}. $$
The Weyl modules $V_n$ 
and their duals $V_n^*$ 
are reducible in general, 
but their composition series are well-known, 
so are the Weyl filtrations of the tilting modules $T_n$. 

\begin{lemma} \label{t}
Write $n=n_0 +\ell n_1$
with $0\leq n_0 \leq \ell-1, n_1 \geq 0$, 
then 
\begin{enumerate}
\item 
if $n_1 = 0$ or $n_0 = \ell - 1$, 
$V_n$ is irreduible, 
hence $T_n = V_n = V_n^* = L_n$; 
\item 
assume now that $0 \leq n_0 \leq \ell - 2$ and $n_1 \geq 1$, 
set $n' = ( \ell- 2 - n_0 ) + \ell (n_1 - 1 ) $, 
then we have the following exact sequences: 
$L_{n'}  \hookrightarrow V_n \twoheadrightarrow L_n $, 
$L_n  \hookrightarrow V_n^* \twoheadrightarrow L_{n'} $, 
$V_n \hookrightarrow T_n  \twoheadrightarrow V_{n'} $ and 
$V_{n'}^* \hookrightarrow T_n \twoheadrightarrow V_n^*$. 
\end{enumerate}
\end{lemma}

\begin{proof}
See [L2, Proposition 9.2] or [APW, Corollary 4.6] for assertions about $V_n, V_n^*$. 
Since $ - \omega_0 n = - (-1) n = n $, 
we have $L_n^* \cong L_n$ and $T_n^* = T_n$. 
By [A2, Proposition 5.8], 
$T_n$ is the projective cover of $L_{n'}$. 
Since $\text{Ext}_{\mathscr C_f}^i ( V_m, V_k^* ) = \Q(q)$
if $i = 0$ and $m = k$;
$0$ otherwise, 
we have the following reciprocity of multiplicities: 
$( T_n, V_k ) = \text{dim}_{\Q(q)} \text{Hom}_{\mathscr C_f} ( T_n, V_k^* ) = ( V_k^*, L_{n'} )$. 
Hence the composition factors of the Weyl modules 
imply the Weyl filtrations of the tilting modules. 
\end{proof}

Let $\mathcal W \cong \Z_2 = \{ 1, -1 \}$ 
be the Weyl group of $\mathfrak {sl}_2$, 
and let $\mathcal W_\ell \cong \Z_2 \ltimes \Z$ 
be the affine Weyl group. 
The shifted action of $\mathcal W_\ell$
on $X = \Z$ 
is defined by:
$(1, m) \cdot n = n +  2m \ell$; 
$( -1, m) \cdot n = - n - 2 + 2 m \ell$.
The fundamental domain for $\mathcal W_\ell$ 
is given by $\bar C = \{ -1, 0, \cdots, \ell - 1 \}$, 
and the linkage principal yields the following decomposition of $O_q$.

\begin{prop}
$O_q = ( \oplus_{k \geq 1} V_{k \ell -1 } \otimes V_{k \ell -1 } ) \oplus ( \oplus_{ m = 0}^{ \ell -2 } \Lambda_m )$ 
as a $U_q \times U_q$-module, 
where $\Lambda_m = \sum_{s \in \mathcal W_\ell \cdot m, s \geq 0} \M ( T_s )$.
\end{prop}

\begin{proof}
See Proposition \ref{link} and Lemma \ref{t}. 
\end{proof}

To analyze the structure of $\Lambda_m$, $0 \leq m \leq \ell - 2$ 
even further, let us take a closer look at the bimodule structure of each $\M ( T_s )$. 

We say that $n_1 < \cdots < n_i < n_{i+1} < \cdots $ 
is a sequence if $n_i = n_{i+1}' $ 
for any $ i \geq 1$
($n_i \geq 0, n_i \neq - 1 \text{   mod   } \ell$ is assumed). 
We can form $\ell -1$ sequences of infinite length starting with $0, 1, \cdots, \ell -2$ respectively, 
which is the same as to arrange the weights 
in the $\mathcal W_\ell$-orbits of $0, 1, \cdots, \ell-2$ in an increasing order.

A module of finite length is called rigid if the socle and radical series coincide, 
in which case the unique shortest filtration with semisimple quotients is called the Loewy series. 
We represent the structure of rigid modules pictorially, 
with the top blocks corresponding to the tops of the modules
and the bottom blocks representing the socles. 

\begin{lemma} 
Let $n_1 < n_2$ be a sequence, i.e. $n_1 = n_2'$, then 
\begin{enumerate}
\item  
$\M ( L_{n_1} ) = L_{n_1} \otimes L_{n_1}$ and $ \M( L_{n_2} ) = L_{n_2} \otimes L_{n_2}$. 
\item  
$\M( V_{n_2} )$ (resp. $\M ( V_{n_2}^* )$)  
is rigid and the Loewy series is given by 
$0 \subset \M(L_{n_1}) \oplus \M(L_{n_2}) \subset \M (V_{n_2})$
(resp.  
$0 \subset \M(L_{n_1}) \oplus \M(L_{n_2}) \subset \M (V_{n_2}^*)$)
with layers depicted by 
$$\M( V_{n_2})  \ \sim \ \begin{gathered}
\xymatrix@R=2pt@C=-20pt{
& *+[F-,]{L_{n_2} \otimes L_{n_1}} &\\
*+[F-,]{L_{n_1} \otimes L_{n_1}} \ar@{}[rr]|-\bigoplus && *+[F-,]{L_{n_2} \otimes L_{n_2}} \\
}\end{gathered}  $$
(resp. 
$$\M( V_{n_2}^*) \ \sim \ \begin{gathered}
\xymatrix@R=2pt@C=-20pt{
& *+[F-,]{L_{n_1} \otimes L_{n_2}} &\\
*+[F-,]{L_{n_1} \otimes L_{n_1}} \ar@{}[rr]|-\bigoplus && *+[F-,]{L_{n_2} \otimes L_{n_2}} \\
}\end{gathered} \ ). $$
\end{enumerate}
\end{lemma}

\begin{proof}
(1) is obvious. 
Tensoring $ 0 \subset L_{n_1} \subset V_{n_2}$ together with 
$ 0 \subset L_{n_2} = \text{Ann } (L_{n_1}) \subset V_{n_2}^*$, 
we obtain a filtration of $V_{n_2} \otimes V_{n_2}^*$: 
$ 0 \subset L_{n_1} \otimes \text{Ann } ( L_{n_1}) 
\subset L_{n_1} \otimes V_{n_2}^* + V_{n_2} \otimes \text{Ann } (L_{n_1}) 
\subset V_{n_2} \otimes V_{n_2}^*$. 
Recall the $U_q \times U_q$-map 
$\phi_{V_{n_2}}: V_{n_2} \otimes V_{n_2}^* \twoheadrightarrow \M (V_{n_2})$, 
it is easy to see that $\text{Ker } \phi_{V_{n_2}} = L_{n_1} \otimes \text{Ann } (L_{n_1})$; 
$\phi_{V_{n_2}} ( L_{n_1} \otimes V_{n_2}^*) = \M (L_{n_1})$; and 
$\phi_{V_{n_2}} ( V_{n_2} \otimes \text{Ann } (L_{n_1}) ) = \M ( L_{n_2} )$. 
It follows that $\M ( V_{n_2} )$ 
admits the filtration as claimed in (2). 
The constituent $L_{n_2} \otimes L_{n_1}$ 
is nontrivially linked with both  $ L_{n_1} \otimes L_{n_1} $ and $ L_{n_2} \otimes L_{n_2}$, 
since the exact sequence $L_{n_1} \hookrightarrow V_{n_2} \twoheadrightarrow L_{n_2}$
does not split.
Similar arguments apply to $\M (V_{n_2}^*)$. 
\end{proof}

\begin{lemma} \label{t1}
Let $n_1 < n_2$ be a sequence and $ 0 \leq n_1 \leq \ell - 2 $, then 
\begin{enumerate}
\item $ \M ( T_{n_2} ) $ is rigid and indecomposable as a $U_q \times U_q$-module. 
The Loewy series is given by 
$0 \subset \M (L_{n_1} ) \oplus \M ( L_{n_2}) \subset \M (V_{n_2}) + \M (V_{n_2}^*) \subset \M (T_{n_2})$ with layers depicted by 
$$\M( T_{n_2})  \ \sim \ \begin{gathered}
\xymatrix@R=2pt@C=-20pt{
& *+[F-,]{L_{n_1}\otimes L_{n_1}}  &\\
*+[F-,]{L_{n_2} \otimes L_{n_1}} \ar@{}[rr]|-\bigoplus && *+[F-,]{L_{n_1} \otimes L_{n_2}} \\
*+[F-,]{L_{n_1} \otimes L_{n_1}} \ar@{}[rr]|-\bigoplus && *+[F-,]{L_{n_2} \otimes L_{n_2}}
}\end{gathered} \ . $$
\item $\M (T_{n_1}) \subset \M (T_{n_2})$ 
and $\M (T_{n_2}) / \M (T_{n_1}) \cong V_{n_2}^* \otimes V_{n_2}^*$. 
\end{enumerate}
\end{lemma}

\begin{proof}
Tensoring $ 0 \subset L_{n_1} \subset V_{n_2} \subset T_{n_2}$ 
together with $ 0 \subset \text{Ann } (V_{n_2}) \subset \text{Ann } (L_{n_1}) \subset T_{n_2}^*$
gives a filtration of $T_{n_2} \otimes T_{n_2}^*$; 
applying $\phi_{T_{n_2}}: T_{n_2} \otimes T_{n_2}^* \twoheadrightarrow \M (T_{n_2})$
to it, we obtain the desired filtration for $ \M (T_{n_2})$. 
Choose a basis of $T_{n_2}$ 
so that the matrix representations with respect to this basis look like 
$$\left( \begin{array}{ccc}  \M (L_{n_1}) & \vartriangle & \bigstar\\ 0 & \M( L_{n_2}) & \triangledown \\ 0&0& \M (L_{n_1}) \end{array} \right). $$
The diagonal blocks correspond to the simple layers of $T_{n_2}$; 
the matrix coefficients $\M (L_{n_1})$ and $\M (L_{n_2})$, 
together with $\vartriangle$ (resp. $\triangledown$), 
span $\M (V_{n_2})$ (resp. $\M (V_{n_2}^*)$); 
the coefficients in $\bigstar$ 
generate the whole $\M (T_{n_2})$. 
It is not hard to see that $\M (T_{n_2})$ 
is indeed rigid and indecomposable. 

It is obvious that $\M (T_{n_1}) \subset \M (T_{n_2})$
since $T_{n_1} \cong L_{n_1}$ for $0 \leq n_1 \leq \ell -2$. 
Moreover $\M (T_{n_2}) / \M (T_{n_1}) 
\cong ( T_{n_2} \otimes T_{n_2}^* ) / \phi_{T_{n_2}}^{-1} \M (L_{n_1}) 
= ( T_{n_2} \otimes T_{n_2}^* ) / (L_{n_1} \otimes T_{n_2}^* + T_{n_2} \otimes \text{Ann } (V_{n_2}))
\cong V_{n_2}^* \otimes V_{n_2}^*$. 
\end{proof}

\begin{lemma} \label{t2}
Let $n_1 < n_2 < n_3$ be a sequence, then 
\begin{enumerate}
\item 
$\M (T_{n_3})$ is rigid and indecomposable. 
The Loewy series is given by 
$0 \subset \M (L_{n_1}) \oplus \M (L_{n_2}) \oplus \M (L_{n_3}) 
\subset \M (V_{n_2}) + \M (V_{n_2}^*) + \M (V_{n_3}) + \M (V_{n_3}^*)
\subset \M (T_{n_3})$
with layers depicted by 
$$M( T_{n_3})  \ \sim \ \begin{gathered}
\xymatrix@R=2pt@C=-4pt{
&& *+[F-,]{L_{n_{2}}\otimes L_{n_2}}  &\\
*+[F-,]{L_{n_2} \otimes L_{n_1}}   \ar@{}[rr]|-\bigoplus & \quad & 
*+[F-,]{L_{n_1} \otimes L_{n_2}}   \ar@{}[rr]|-\bigoplus & \quad  &   
*+[F-,]{L_{n_3} \otimes L_{n_2}}  \ar@{}[rr]|-\bigoplus & \quad &  
*+[F-,]{L_{n_2} \otimes L_{n_3}} \\
*+[F-,]{L_{n_1} \otimes L_{n_1}}   \ar@{}[rr]|-\bigoplus  & \quad &
*+[F-,]{L_{n_2} \otimes L_{n_2}}   \ar@{}[rr]|-\bigoplus & \quad &
*+[F-,]{L_{n_3} \otimes L_{n_3}}  &
}\end{gathered} \ . $$
\item $\M (T_{n_2}) \cap \M (T_{n_3})  = \M (V_{n_2}) + \M (V_{n_2}^*)$. 
Moreover we have 
$\M (T_{n_3}) / ( \M (T_{n_2}) \cap \M (T_{n_3}) ) \cong V_{n_3}^* \otimes V_{n_3}^*$ and 
$\M (T_{n_3}) / ( \M (V_{n_3}) + \M (V_{n_3}^*) )  \cong V_{n_2} \otimes V_{n_2}$. 
\end{enumerate} 
\end{lemma}

\begin{proof}
The proof is parallel to the proof of the previous two lemmas. 
Since $T_{n_3}$ is rigid with layers 
$L_{n_2}$, $L_{n_1} \oplus L_{n_3}$ and $L_{n_2}$
from the socle to the top, 
we can choose a basis of $T_{n_3}$ 
so that the matrix representations with respect to this basis look like 
$$\left( \begin{array}{cccc}  \M ( L_{n_2} ) & \vartriangle_{n_3} & \triangledown_{n_2} & \bigstar \\ 
0& \M ( L_{n_3} ) & 0 & \triangledown_{n_3} \\ 
0 & 0 & \M ( L_{n_1} )  & \vartriangle_{n_2} \\ 
0 & 0 & 0 & \M ( L_{n_2} ) 
\end{array} \right). $$
We have the matrix coefficients of the irreducibles on the diagonal; 
$\vartriangle_{n_3}$ (resp. $\triangledown_{n_3} $) 
together with $\M (L_{n_2})$, $\M ( L_{n_3} )$ 
span $\M ( V_{n_3} ) $ (resp. $\M ( V_{n_3}^* ) $); 
$\vartriangle_{n_2}$ (resp. $\triangledown_{n_2} $)
together with $\M ( L_{n_1} )$, $\M ( L_{n_2} )$ 
span $\M ( V_{n_2} ) $ (resp. $\M ( V_{n_2}^* )$); 
the top $\bigstar$ generates the whole $\M (T_{n_3} )$. 
Again it is not difficult to see that $\M (T_{n_3})$ 
is rigid and indecomposable, 
and the nonzero blocks in the matrix correspond to the layers of the Loewy series. 

It's clear that 
$\M (T_{n_2}) \cap \M (T_{n_3})  = \M (V_{n_2}) + \M (V_{n_2}^*)$. 
Since 
$\phi_{T_{n_3}}^{-1} ( \M (V_{n_2} ) + \M (V_{n_2}^*) ) 
= V_{n_2}^* \otimes T_{n_3}^* + T_{n_3} \otimes \text{Ann } (V_{n_3} )$ 
and 
$\phi_{T_{n_3}}^{-1} ( \M ( V_{n_3} ) + \M (V_{n_3}^*) ) 
= V_{n_3} \otimes T_{n_3}^* + T_{n_3} \otimes \text{Ann } (V_{n_2}^*) $, 
the last two isomorphisms hold. 
\end{proof}

\begin{theorem} \label{maintheorem1}
Let $n = n_1 < n_2 < \cdots < n_i < \cdots $ 
be the sequence of infinite length starting at $n$
for $0 \leq n \leq \ell -2$. 
\begin{enumerate}
\item $\Lambda_n$ is rigid and indecomposable as a $U_q \times U_q$-module. 
The Loewy series is given by 
$$ 0 \subset \oplus_{i \geq 1} \M (L_{n_i})
\subset \sum_{i \geq 1} \M (V_{n_i}) + \sum_{i \geq 1} \M (V_{n_i}^* )
\subset \sum_{i \geq 1} \M (T_{n_i}) = \Lambda_n $$
with layers
$\oplus_{i \geq 1} L_{n_i} \otimes L_{n_i}$, 
$\oplus_{i \geq 1} ( L_{n_{i+1}} \otimes L_{n_i} \oplus L_{n_i} \otimes L_{n_{i+1}} )$ 
and 
$\oplus_{i \geq 1} L_{n_i} \otimes L_{n_i}$. 

\vspace{.05in}

\item $\Lambda_n$ also admits an increasing filtration of $U_q \times U_q$-submodules 
$$ 0 = P^0 \subset P^1 \subset \cdots \subset P^i \subset \cdots $$
and a decreasing filtration of $U_q \times U_q$-submodules 
$$ \cdots \subset Q^i \subset \cdots \subset Q^2 \subset Q^1 \subset Q^0 = \Lambda_n$$
such that 
$\cup_i P^i = \Lambda_n$, 
$P^i / P^{i-1} \cong V_{n_i}^* \otimes V_{n_i}^*$, 
and 
$\cap_i Q^i = 0$, 
$Q^{i-1} / Q^i \cong V_{n_i} \otimes V_{n_i}$. 
\end{enumerate}
\end{theorem}

\begin{proof}
It follows from the three lemmas. 
For (2), set $P^i = \sum_{j \leq i} \M (T_{n_j})$ 
and $Q^i = \sum_{j \geq i+2} \M (T_{n_j})$. 
\end{proof}

Finally let's find out explicitly the cocommutative elements of $O_q$. 
Let $Y_n \subset O_q$ be the linear span of monomials $a^m b^k c^h, b^k c^h d^l$
of degree $\leq n$, 
i.e. $Y_n = \sum_{i \leq n} \M ( T_i ) = \sum_{i \leq n} \M ( T_1)^i$.

\begin{lemma}
$Y_0 \subset Y_1 \subset \cdots \subset Y_{n-1} \subset Y_n \subset \cdots$
is a filtration of $U_q \times U_q$-submodules of $O_q$ 
with subquotients $Y_n / Y_{n-1} \cong V_n^* \otimes V_n^*$. 
\end{lemma}

\begin{proof}
It follows from Lemma \ref{t1} (2) and Lemma \ref{t2} (2). 
\end{proof}

\begin{lemma} \label{t3}
The subspace $\{ x \in V_n^* \otimes V_n^*: \rho_1' (u) x = \rho_2' ( S^{-1} u ) x,  \forall u \in U_q \}$ 
is one-dimensional 
where $\rho_1', \rho_2'$ denote the actions of $U_q$ on the two copies of $V_n^*$. 
\end{lemma}

\begin{proof}
Of course it follows from Lemma \ref{dim1}, the general version of it. 
But here we can compute more explicitly, 
which is actually the motivation behind the proof of Lemma \ref{dim1}. 

Recall that $V_n^*$ is $(n+1)$-dimensional and has a basis $e_0, e_1, \cdots, e_n$ such that 
$e_i$ is of weight $n - 2 i$ and 
$E^{(j)} e_i =  \left[ \begin{array}{c} i \\ j  \end{array} \right]_q e_{i-j}$; 
$F^{(j)} e_i =  \left[ \begin{array}{c} n-i \\ j \end{array} \right]_q e_{i+j}$. 
For any $x = \sum_{i, j} x_{ij} e_i \otimes e_j \in V_n^* \otimes V_n^*$, 
if it satisfies that $\rho_1' (u^0 ) x = \rho_2' ( S^{-1} u^0 ) x$ 
for any $u^0 \in U_q^0$, 
we must have $x_{ij} = 0$ 
except for $ i + j = n$. 
Now let $x = \sum_i x_{i, n-i} e_i \otimes e_{n-i}$, 
since $S^{-1} E^{(j)} = (-1)^j q^{-j(j-1)} E^{(j)} K^{-j}$, 
it follows that 
$$\rho_1' (E^{(j)}) x 
= \sum_{i = j}^n x_{i, n-i} \left[ \begin{array}{c} i \\  j \end{array} \right]_q e_{i-j} \otimes e_{n-i}$$
and 
$$\rho_2' (S^{-1} E^{(j)}) x
= \sum_{i =0}^{n-j} x_{i, n-i} (-1)^j q^{-j(j-1)} q^{-j (2i-n)} \left[ \begin{array}{c} n-i \\ j  \end{array} \right]_q e_i \otimes e_{n-i-j}$$
for any $j \in \N, j \leq n$. 
Hence if $\rho_1' (E^{(j)}) x = \rho_2' (S^{-1} E^{(j)}) x$, 
then 
$x_{i, n-i} \left[ \begin{array}{c} i \\ j \end{array} \right]_q 
= x_{i-j, n-i+j} (-1)^j q^{j(j+n-2i+1)} \left[ \begin{array}{c} n-i+j \\ j  \end{array} \right]_q$, 
in particular  
$x_{i, n-i} = x_{0,n} (-1)^i q^{i(n-i+1)} \left[ \begin{array}{c} n \\ i  \end{array} \right]_q$, 
which implies that  $x = x_{0, n} y$  
with  
$y = \sum_{i=0}^n (-1)^i q^{i(n-i+1)} \left[ \begin{array}{c} n \\ i  \end{array} \right]_q e_i \otimes e_{n-i}$. 
On the other hand it is straightforward to check that  
$\rho_1' (u) y = \rho_2' (S^{-1} u) y$ holds for any $u\in U_q$. 
\end{proof}

\begin{prop}
$HH^0 (O_q, O_q) = \Q(q) [a+d]$.
\end{prop}

\begin{proof}
Denote the set of cocommutative elements of $O_q$ by $\mathbf {Co}$.
We need to show that $\mathbf{Co}$ consists of polynomials in $a+d$.
Recall that $O_q = \cup_{n \geq 0} Y_n$, 
and it is trivial that $Y_0 \cap \mathbf {Co} = \Q(q)$. 
Assume now that $Y_n \cap \mathbf {Co} $
is linearly spanned by polynomials of degree $\leq n$ in $a+d$. 
Suppose $f \in Y_{n+1} \cap \mathbf {Co}$, 
then $\rho_1 (u) f = \rho_2 (S^{-1} u) f $ for any $u\in U_q$. 
Since the image of $(a+d)^{n+1}$ in $Y_{n+1} / Y_n$ is nonzero, 
by Lemma \ref{t3} there exists a scalar $\zeta$ 
such that $f - \zeta (a + d)^{n+1} \in Y_n$. 
Note that $f - \zeta (a + d)^{n+1}$ is also cocommutative, 
i.e. it belongs to $Y_n \cap \mathbf {Co}$, 
by induction $f - \zeta (a + d)^{n+1}$ is a polynomial of degree $\leq n$
in $a+d$, therefore $f$ is a polynomial of degree $\leq n+1$ in $a+d$. 
\end{proof}

\end{document}